\documentclass[11pt,reqno]{amsart}

\usepackage{fullpage, amsmath, amsthm, amstext, amssymb, amsfonts, fancyhdr, graphicx, tikz, forest, mathrsfs, array, mathtools, bm}
\usepackage[normalem]{ulem}
\usetikzlibrary{graphs,graphs.standard,quotes}
\usetikzlibrary{matrix,positioning,quotes}

\definecolor{myblue}{RGB}{80,80,160}
\definecolor{mygreen}{RGB}{80,160,80}

\numberwithin{equation}{section}

\newtheorem{theorem}{Theorem}[section]
\newtheorem{lemma}[theorem]{Lemma}

\newtheorem{corollary}[theorem]{Corollary}

\newtheorem{example}[theorem]{Example}
\newtheorem{definition}[theorem]{Definition}
\newtheorem{proposition}[theorem]{Proposition}

\theoremstyle{remark}

{\bf}{\rm}

\DeclareMathOperator{\PF}{\text{PF}}

\newcommand{\aR}{\mathbf a}
\newcommand{\bR}{\mathbf b}

\newcommand{\sR}{\mathbf s}

\newcommand{\uR}{\mathbf u}
\newcommand{\vR}{\mathbf v}
\newcommand{\wR}{\mathbf w}

\newcommand{\piR}{\boldsymbol \pi}
\newcommand{\tauR}{\boldsymbol \tau}
\newcommand{\lambdaR}{\boldsymbol \lambda}
\newcommand{\alphaR}{\boldsymbol \alpha}

\newcommand{\MS}{\text{MS}}
\newcommand{\disp}{\text{disp}}

\newcommand{\PR}{\mathbb P}
\newcommand{\QR}{\mathbb Q}
\newcommand{\ER}{\mathbb E}

\newcommand{\Var}{\mathbb V \mathrm{ar}}
\newcommand{\Cov}{\mathbb C \mathrm{ov}}
\newcommand{\be}{\begin{equation}}
\newcommand{\ee}{\end{equation}}

\newcommand{\old}[1]{}

\NewDocumentCommand\DownArrow{O{2.0ex} O{black}}{%
   \mathrel{\tikz[baseline] \draw [<->, line width=0.5pt, #2] (0,0) -- ++(0,#1);}
}

\begin{document}

\forestset{parent color/.style args={#1}{
    {fill=#1},
    for tree={fill/. wrap pgfmath arg={#1!##1}{1/level()*80},draw=#1!80!black}},
    root color/.style args={#1}{fill={{#1!60!black!25},draw=#1!80!black}}
}

\title{Parking functions, multi-shuffle, and asymptotic phenomena}

\author{Mei Yin}
\thanks{MY was partially supported by the University of Denver's Faculty Research Fund 84688-145601.}
\address{Department of Mathematics, University of Denver, Denver, CO 80208}
\email{mei.yin at du.edu}

\date{\today}


\begin{abstract}
Given a positive-integer-valued vector $\uR=(u_1, \dots, u_m)$ with $u_1<\cdots<u_m$. A $\uR$-parking function of length $m$ is a sequence $\piR=(\pi_1, \dots, \pi_m)$ of positive integers whose non-decreasing rearrangement $(\lambda_1, \dots, \lambda_m)$ satisfies $\lambda_i\leq u_i$ for all $1\leq i\leq m$. We introduce a combinatorial construction termed a parking function multi-shuffle to generic $\uR$-parking functions and obtain an explicit characterization of multiple parking coordinates. As an application, we derive various asymptotic probabilistic properties of a uniform $\uR$-parking function when $u_i=cm+ib$. The asymptotic scenario in the generic situation $c>0$ is in sharp contrast with that of the special situation $c=0$.
\end{abstract}

\maketitle

\section{Introduction}
\label{intro}
Parking functions were introduced by Konheim and Weiss \cite{KW}, under the name of ``parking disciplines,'' to study the following problem. Consider a parking lot with $n$ parking spots placed sequentially along a one-way street. In order, a line of $m \leq n$ cars enters the lot. The $i$th car drives to its preferred spot $\pi_i$ and parks there if possible, and otherwise takes the next available spot if it exists. The sequence of preferences $\piR=(\pi_1,\dots,\pi_m)$ is called a \emph{parking function} if all cars successfully park. We denote the set of parking functions by $\PF(m, n)$, where $m$ is the number of cars and $n$ is the number of parking spots.

Nowadays, parking functions are an established area of research in combinatorics, with connections to labeled trees and forests (Chassaing and Marckert, \cite{CM}), hyperplane arrangements, interval orders, and plane partitions (Stanley, \cite{Stanley1} \cite{Stanley2}), diagonal harmonics and $(q, t)$-analogs of Catalan numbers (Haglund et al., \cite{Haglund}), abelian sandpiles (Cori and Rossin, \cite{CR}), to mention a few. Properties of random parking functions have also been of interest to statisticians and probabilists (Diaconis and Hicks, \cite{DH}). We refer to Knuth \cite[Section 6.4]{Knuth2} and Yan \cite{Yan} for a comprehensive survey.

Given a positive-integer-valued vector $\uR=(u_1, \dots, u_m)$ with $u_1<\cdots<u_m$. A \emph{$\uR$-parking function} of length $m$ is a sequence $\piR=(\pi_1, \dots, \pi_m)$ of positive integers whose non-decreasing rearrangement $(\lambda_1, \dots, \lambda_m)$ satisfies $\lambda_i\leq u_i$ for all $1\leq i\leq m$. Denote the set of $\uR$-parking functions by $\PF(\uR)$. There is a similar interpretation for $\uR$-parking functions in terms of the parking scenario depicted above: One wishes to park $m$ cars in a one-way street with $u_m$ spots, but only $m$ spots, at positions $u_1, \dots, u_m$, are still empty \cite{KuYa}. We recognize that the parking function $\PF(m, n)$ is a special case of the more general $\uR$-parking functions with $u_i=n-m+i$.

In our previous work on $\PF(m, n)$ \cite{KY} \cite{Yin}, we introduced an original combinatorial construction which we term a \emph{parking function multi-shuffle}, and it facilitated an investigation of the properties of a parking function chosen uniformly at random from $\PF(m, n)$. This paper will delve into the essence of the multi-shuffle combinatorial construction on parking functions and introduce the concept to generic $\uR$-parking functions, thus allowing for an explicit characterization of multiple coordinates of $\uR$-parking functions. As an application, we will derive various asymptotic probabilistic properties of a uniform $(a, b)$-parking function, which is a class of $\uR$-parking functions where $u_i=a+(i-1)b$ for some positive integers $a$ and $b$. Denote the set of $(a, b)$-parking functions of length $m$ by $\PF(a, b, m)$. It coincides with $\PF(m, n)$ when $a=n-m+1$ and $b=1$. In view of and generalizing this correspondence, we will study the asymptotics of $\PF(a, b, m)$ when $b \geq 1$ is any integer and $a=cm+b$ for some $c \geq 0$. We find that for large $m$, various probabilistic quantities display strikingly different asymptotic tendency in the generic situation $c>0$ (corresponding to $m \lesssim n$) vs. the special situation $c=0$ (corresponding to $m=n$), including the boundary behavior of a single coordinate and all moments of multiple coordinates.

Our asymptotic calculation utilizes the multi-dimensional Cauchy product of the \emph{tree function} $F(z)$, which is a variant of the Lambert function, the \emph{Gon\v{c}arov polynomials} $g_m(x; a_0,  a_1, \dots, a_{m-1})$ for $m=0, 1, \dots$, which form a natural basis for working with $\uR$-parking functions, as well as \emph{Abel's multinomial theorem}. In particular, our new perspective on parking functions leads to asymptotic moment calculations for multiple coordinates of $(a, b)$-parking functions that complement the work of Kung and Yan \cite{KY2}, where the explicit formulas for the first and second factorial moments and a general form for the higher factorial moments of sums of $(a, b)$-parking functions were given.

This paper is organized as follows. Section \ref{pfs} illustrates the notion of \emph{$\uR$-parking function multi-shuffle} that decomposes a $\uR$-parking function into smaller components (Definition \ref{shuffle}). This construction offers an explicit characterization of multiple coordinates of $\uR$-parking functions (Theorems \ref{main1} and \ref{component}). Theorem \ref{component} enumerates $\uR$-parking functions in connection with Gon\v{c}arov polynomials, but we also provide an alternative description of $\uR$-parking functions in Proposition \ref{un2}. Section \ref{random} uses the multi-shuffle construction introduced in Section \ref{pfs} to investigate various properties of a parking function chosen uniformly at random from $\PF(cm+b, b, m)$. When the parking preferences $\pi_1, \dots, \pi_l$ are exactly $b$ spots apart, a simplified characterization of $(cm+b, b)$-parking functions is given in Section \ref{generic-case}. Building upon Theorem \ref{component} and Proposition \ref{un}, we compute asymptotics of all moments of multiple coordinates in Theorem \ref{general-mean} in the generic situation $c>0$ and give complete technical details for all moments of two coordinates (Theorem \ref{mean}). The asymptotic mixed moments in the generic situation $c>0$ is contrasted with that of the special situation $c=0$ in Section \ref{special}. We then focus on the boundary behavior of a single coordinate in Section \ref{section-boundary}. We find that in the generic situation $c>0$ on the right end it approximates a Borel distribution with parameter $b/(b+c)$ while on the left end it deviates from the constant value in a rescaled Poisson fashion  (Corollaries \ref{boundary2} and \ref{boundary1}). This asymptotic tendency differs from that in the special situation $c=0$, where the boundary behavior of a single coordinate on the left and right ends both approach Borel$(1)$ (Corollaries \ref{boundary2} and \ref{boundary1-special}).

\subsection*{Notations} Let $\mathbb{N}$ be the set of positive integers. For $m, n \in \mathbb{N}$, we write $[m, n]$ for the set of integers $\{m, \dots, n\}$ and $[n]=[1, n]$. For vectors $\aR, \bR \in [n]^m$, denote by $\aR\leq_C \bR$ if $a_i\leq b_i$ for all $i\in [m]$; this is the component-wise partial order on $[n]^m$. In a similar fashion, denote by $\aR <_C \bR$ if $a_i\leq b_i$ for all $i\in [m]$ and there is at least one $j \in [m]$ such that $a_j<b_j$. For $\bR \in [n]^m$, we write $[\bR]$ for the set of $\aR \in [n]^m$ with $\aR \leq_C \bR$.

\section{$u$-parking function multi-shuffle}
\label{pfs}
In this section we explore the properties of generic $\uR$-parking functions through a $\uR$-parking function multi-shuffle construction. We will write our results in terms of parking coordinates $\pi_1, \dots, \pi_l$ for explicitness, where $1\leq l \leq m$ is any integer. But due to permutation symmetry, they may be interpreted for any coordinates. Temporarily fix $\pi_{l+1}, \dots, \pi_m$. Let
\begin{equation}
A_{\pi_{l+1}, \dots, \pi_m}=\{\vR=(v_1, \dots, v_l): (v_1, \dots, v_l, \pi_{l+1}, \dots, \pi_m)\in \PF(\uR)\},
\end{equation}
where $\vR$ is in non-decreasing order.

\begin{proposition}\label{nonempty}
Take $1\leq l \leq m$ any integer. Suppose that $\vR$ is in non-decreasing order and maximally compatible (in component-wise partial order) with the fixed $\pi_{l+1}, \dots, \pi_m$. Then for all $1\leq i\leq l$, $v_i=u_{k_i}$ with $1\leq k_1<\cdots<k_l\leq m$, and $v_i \neq \pi_{l+1}, \dots, \pi_m$.
\end{proposition}

\begin{proof}
Take $\piR=(v_1, \dots, v_l, \pi_{l+1}, \dots, \pi_m)$ a $\uR$-parking function. Let $\lambdaR=(\lambda_1, \dots, \lambda_m)$ be the non-decreasing rearrangement of $\piR$. Temporarily fix an arbitrary $i$, where $1\leq i\leq l$. Then for some $k$, $v_i=\lambda_k \leq u_k$. If there are multiple entries in $\piR$ (and hence $\lambdaR$) that equal $v_i$, we may assume that $k$ is the maximum such index, which implies that $\lambda_k<\lambda_{k+1}$ if $k<m$. We claim that $v_i=u_k$. Suppose otherwise that $v_i<u_k$. Then $v_i+1=\lambda_k+1 \leq u_k$ and $(\lambda_1, \dots, \lambda_{k-1}, \lambda_k+1, \lambda_{k+1}, \dots, \lambda_m)$ is the non-decreasing rearrangement of $\piR^i=(v_1, \dots, v_{i-1}, v_i+1, v_{i+1}, \dots, v_l, \pi_{l+1}, \dots, \pi_m)$, making $\piR^i$ a $\uR$-parking function. This contradicts the assumption that $\vR$ is maximal. Since $u_1<\cdots<u_m$, we further have $v_i$ is the unique entry in $\piR$ with $v_i=u_k$, as $v_i>u_j$ for any $j<k$.
\end{proof}

To identify the maximal $\vR$ in $A_{\pi_{l+1}, \dots, \pi_m}$, we arrange $\pi_i$ for $l+1\leq i\leq m$ in non-decreasing order, denoted by $\pi_{(l+1)} \leq \cdots \leq \pi_{(m)}$. Set $n_l=0$. We find the minimum index $n_i$ in order, starting with $n_{l+1}$, such that $n_i>n_{i-1}$ and $u_{n_i} \geq \pi_{(i)}$ for each $l+1\leq i\leq m$. If such $u_{n_i}$'s cannot be located, then $A_{\pi_{l+1}, \dots, \pi_m}$ is empty. Otherwise excluding these $u_{n_i}$'s from $\uR$ gives the optimal $\vR$. From the parking scheme, if $\vR\in A_{\pi_{l+1}, \dots, \pi_m}$, then $\wR\in A_{\pi_{l+1}, \dots, \pi_m}$ for all $\wR\leq_C \vR$, where $\leq_C$ is the component-wise partial order. This implies that if $A_{\pi_{l+1}, \dots, \pi_m}$ is non-empty, then there is a unique maximal element (in component-wise partial order) $\vR \in [u_m]^l$ with $v_i=u_{k_i}$ for all $1\leq i\leq l$, where $1\leq k_1<\cdots<k_l\leq m$, and $A_{\pi_{l+1}, \dots, \pi_m}=[\vR]$. Therefore given the last $m-l$ parking preferences, it is sufficient to identify the largest feasible first $l$ preferences (if exists).

\begin{example}
Take $\uR=(u_1, u_2, u_3, u_4)=(2, 3, 5, 8)$, $\pi_3=6$, and $\pi_4=2$. Then $A_{\pi_3, \pi_4}=[\vR]=[(u_2, u_3)]=[(3, 5)]$. See illustration below.
\begin{equation*}
\begin{array}{ccccc}
\pi_{(3)} & 2 & \leq & 2 & u_1\\
v_1 & &  & 3 & u_2 \\
v_2 & &  & 5 & u_3 \\
\pi_{(4)} & 6 & \leq & 8 & u_4
\end{array}
\end{equation*}
\end{example}

We will now introduce an original combinatorial construction which we term a parking function multi-shuffle to generic $\uR$-parking functions.

\begin{definition}[{$\uR$-parking function multi-shuffle}]\label{shuffle}
Take $1\leq l \leq m$ any integer. Let $\vR=(v_1, \dots, v_l) \in [u_m]^l$ be in increasing order with $v_i=u_{k_i}$ for all $1\leq i\leq l$, where $1\leq k_1<\cdots<k_l\leq m$. Say that $\pi_{l+1}, \dots, \pi_m$ is a \emph{$\uR$-parking function multi-shuffle} of $l+1$ $\uR$-parking functions $\alphaR_1 \in \PF(u_1, \dots, u_{k_1-1}), \alphaR_2 \in \PF(u_{k_1+1}-u_{k_1}, \dots, u_{k_2-1}-u_{k_1}), \dots, \alphaR_l \in \PF(u_{k_{l-1}+1}-u_{k_{l-1}}, \dots, u_{k_l-1}-u_{k_{l-1}})$, and $\alphaR_{l+1} \in \PF(u_{k_l+1}-u_{k_l}, \dots, u_m-u_{k_l})$ if $\pi_{l+1}, \dots, \pi_m$ is any permutation of the union of the $l+1$ words $\alphaR_1, \alphaR_2+(u_{k_1}, \dots, u_{k_1}), \dots, \alphaR_{l+1}+(u_{k_l}, \dots, u_{k_l})$. (If $k_{j-1}=k_{j}-1$ for some $j$, we take the corresponding $\alphaR_j$ as empty.) We will denote this by $(\pi_{l+1}, \dots, \pi_m) \in \MS(\vR)$.
\end{definition}

\begin{example}\label{example}
Take $\uR=(u_1, u_2, u_3, u_4, u_5, u_6, u_7, u_8)=(3, 4, 5, 6, 7, 8, 9, 10)$ and $\vR=(u_4, u_6)=(6, 8)$. Take $\alphaR_1=(2, 1, 2) \in \PF(u_1, u_2, u_3)=\PF(3, 4, 5)$, $\alphaR_2=(1) \in \PF(u_5-u_4)=\PF(1)$, and $\alphaR_3=(2, 1) \in \PF(u_7-u_6, u_8-u_6)=\PF(1, 2)$. Then $(2, \overline{7},2, \underline{9}, \underline{10},1) \in \MS(6, 8)$ is a multi-shuffle of the three words $(2, 1, 2)$, $(7)$, and $(10, 9)$.

Take $\uR=(u_1, u_2, u_3, u_4, u_5, u_6, u_7)=(2, 4, 5, 6, 8, 12, 15)$ and $\vR=(u_2, u_5)=(4, 8)$. Take $\alphaR_1=(1) \in \PF(u_1)=\PF(2)$, $\alphaR_2=(1, 2) \in \PF(u_3-u_2, u_4-u_2)=\PF(1, 2)$, and $\alphaR_3=(5, 3) \in \PF(u_6-u_5, u_7-u_5)=\PF(4, 7)$. Then $(\underline{11}, \underline{13}, \overline{5}, 1, \overline{6}) \in \MS(4, 8)$ is a multi-shuffle of the three words $(1)$, $(5, 6)$, and $(13, 11)$.
\end{example}

The $\uR$-parking function multi-shuffle allows for an explicit characterization of multiple coordinates of $\uR$-parking functions. It connects the identification of the maximal element in $A_{\pi_{l+1}, \dots, \pi_m}$ to the decomposition of $\pi_{l+1}, \dots, \pi_m$ into a multi-shuffle.

\begin{theorem}\label{main1}
Take $1\leq l \leq m$ any integer. Let $\vR=(v_1, \dots, v_l) \in [u_m]^l$ be in increasing order with $v_i=u_{k_i}$ for all $1\leq i\leq l$, where $1\leq k_1<\cdots<k_l\leq m$. Then $A_{\pi_{l+1}, \dots, \pi_m}=[\vR]$ if and only if $(\pi_{l+1}, \dots, \pi_m) \in \MS(\vR)$.
\end{theorem}

\begin{proof}
``$\Longrightarrow$'' Take $\piR=(v_1, \dots, v_l, \pi_{l+1}, \dots, \pi_m)$ a $\uR$-parking function, where $\vR=(u_{k_1}, \dots, u_{k_l})$ with $1\leq k_1<\cdots<k_l\leq m$ maximally compatible with the fixed $\pi_{l+1}, \dots, \pi_m$. Let $\lambdaR=(\lambda_1, \dots, \lambda_m)$ be the non-decreasing rearrangement of $\piR$. Temporarily fix an arbitrary $i$, where $1\leq i\leq l$. Then for some $j$, $v_i=u_{k_i}=\lambda_j \leq u_j$. We claim that $k_i=j$. Suppose otherwise that $k_i<j$. Then $v_i+1=\lambda_j+1=u_{k_i}+1 \leq u_j$ since $\uR$ is strictly increasing. From Proposition \ref{nonempty}, the entry $v_i$ is unique in $\lambdaR$, so $v_i=\lambda_j<\lambda_{j+1}$. This says that $(\lambda_1, \dots, \lambda_{j-1}, \lambda_j+1, \lambda_{j+1}, \dots, \lambda_m)$ is the non-decreasing rearrangement of $\piR^i=(v_1, \dots, v_{i-1}, v_i+1, v_{i+1}, \dots, v_l, \pi_{l+1}, \dots, \pi_m)$, making $\piR^i$ a $\uR$-parking function, contradicting the assumption that $\vR$ is maximal. Therefore $v_i$ is the $k_i$-th entry in the non-decreasing rearrangement of $\piR$ for all $1\leq i\leq l$.

Hence excluding the first $l$ cars which take values in $u_{k_1}, \dots, u_{k_l}$, $\piR$ has exactly $k_1-1$ cars with non-decreasing preference $\leq u_1, \dots, \leq u_{k_1-1}$ respectively (name the subsequence $\alphaR_1$), exactly $k_2-k_1-1$ cars with non-decreasing preference $\geq u_{k_1}+1$ and $\leq u_{k_1+1}$, $\dots$, $\geq u_{k_1}+1$ and $\leq u_{k_2-1}$ (name the subsequence $\alphaR_2'$), $\dots$, exactly $k_l-k_{l-1}-1$ cars with non-decreasing preference $\geq u_{k_{l-1}}+1$ and $\leq u_{k_{l-1}+1}$, $\dots$, $\geq u_{k_{l-1}}+1$ and $\leq u_{k_l-1}$ respectively (name the subsequence $\alphaR_l'$), and exactly $m-k_l$ cars with non-decreasing preference $\geq u_{k_l}+1$ and $\leq u_{k_l+1}$, $\dots$, $\geq u_{k_l}+1$ and $\leq u_{m}$ (name the subsequence $\alphaR_{l+1}'$). Construct $\alphaR_2=\alphaR_2'-(u_{k_1}, \dots, u_{k_1}), \dots, \alphaR_{l+1}=\alphaR_{l+1}'-(u_{k_l}, \dots, u_{k_l})$. It is clear from the above reasoning that $\alphaR_1 \in \PF(u_1, \dots, u_{k_1-1}), \alphaR_2 \in \PF(u_{k_1+1}-u_{k_1}, \dots, u_{k_2-1}-u_{k_1}), \dots, \alphaR_l \in \PF(u_{k_{l-1}+1}-u_{k_{l-1}}, \dots, u_{k_l-1}-u_{k_{l-1}})$, and $\alphaR_{l+1} \in \PF(u_{k_l+1}-u_{k_l}, \dots, u_m-u_{k_l})$. By Definition \ref{shuffle}, $(\pi_{l+1}, \dots, \pi_m) \in \MS(\vR)$.

``$\Longleftarrow$'' We first show that $\piR=(v_1, \dots, v_l, \pi_{l+1}, \dots, \pi_m)$ is a $\uR$-parking function. This is immediate, since from Definition \ref{shuffle}, the non-decreasing rearrangement of $\piR$ is a concatenation of $\alphaR_1$, $u_{k_1}$, $\alphaR_2+(u_{k_1}, \dots, u_{k_1})$, $u_{k_2}$, $\dots$, $\alphaR_l+(u_{k_{l-1}}, \dots, u_{k_{l-1}})$, $u_{k_l}$, and $\alphaR_{l+1}+(u_{k_l}, \dots, u_{k_l})$.

Next we show that $\piR^i=(v_1, \dots, v_{i-1}, v_i+1, v_{i+1}, \dots, v_l, \pi_{l+1}, \dots, \pi_m)$ is not a parking function for any $1 \leq i\leq l$. This is also immediate since the non-decreasing rearrangement of $\piR^i$ only differs from the non-decreasing rearrangement of $\piR$ in the $k_i$-th position with value $u_{k_i}+1>u_{k_i}$.

Combining, we have $A_{\pi_{l+1}, \dots, \pi_m}=[\vR]$.
\end{proof}

\begin{example}[Continued from Example \ref{example}] Take $\uR=(3, 4, 5, 6, 7, 8, 9, 10)$, then $A_{2, 7, 2, 9, 10, 1}=[(6, 8)]$ is equivalent to $(2, 7, 2, 9, 10, 1) \in \MS(6, 8)$. Take $\uR=(2, 4, 5, 6, 8, 12, 15)$, then $A_{11, 13, 5, 1, 6}=[(4, 8)]$ is equivalent to $(11, 13, 5, 1, 6) \in \MS(4, 8)$. See illustration below.
\begin{equation*}
\begin{array}{ccccc}
  \pi_{(3)} & 1 & \leq & 3 & u_1 \\
  \pi_{(4)} & 2 & \leq & 4 & u_2 \\
  \pi_{(5)} & 2 & \leq & 5 & u_3 \\
   v_1 &  &  & 6 & u_4 \\
  \pi_{(6)} & 7 & \leq & 7 & u_5 \\
   v_2 &  &  & 8 & u_6 \\
  \pi_{(7)} & 9 & \leq & 9 & u_7 \\
  \pi_{(8)} & 10 & \leq & 10 & u_8
\end{array} \hspace{3cm}
\begin{array}{ccccc}
  \pi_{(3)} & 1 & \leq & 2 & u_1 \\
  v_1 &  & & 4 & u_2 \\
  \pi_{(4)} & 5 & \leq & 5 & u_3 \\
  \pi_{(5)} & 6 & \leq & 6 & u_4 \\
  v_2 & & & 8 & u_5 \\
  \pi_{(6)} & 11 & \leq & 12 & u_6 \\
  \pi_{(7)} & 13 & \leq & 15 & u_7
\end{array}
\end{equation*}
\end{example}

Of relevance to our investigation, we also introduce the Gon\v{c}arov polynomials. Let $(a_0, a_1, \dots)$ be a sequence of numbers. The Gon\v{c}arov polynomials $g_m(x; a_0,  a_1, \dots, a_{m-1})$ for $m=0, 1, \dots$ are the basis of solutions to the Gon\v{c}arov interpolation problem in numerical analysis. They are defined by the biorthogonality relation:
\begin{equation}
\epsilon(a_i) D^i g_m(x; a_0, a_1, \dots, a_{m-1})=m! \delta_{im},
\end{equation}
where $\epsilon(a_i)$ is evaluation at $a_i$, $D$ is the differentiation operator, and $\delta_{im}$ is the Kronecker delta. The Gon\v{c}arov polynomials satisfy many nice algebraic and analytic properties, making them very useful in analysis and combinatorics. Specifically, we list two properties of Gon\v{c}arov polynomials below:
\begin{enumerate}
\item Determinant formula.
\begin{equation}
g_m(x; a_0, a_1, \dots, a_{m-1}) = m! \hspace{.1cm} \left| \begin{matrix}
1 & a_0 & \frac{a_0^2}{2!} & \frac{a_0^3}{3!} & \cdots & \frac{a_0^{m-1}}{(m-1)!} & \frac{a_0^m}{m!} \\
0 & 1 & a_1 & \frac{a_1^2}{2!} & \cdots & \frac{a_1^{m-2}}{(m-2)!} & \frac{a_1^{m-1}}{(m-1)!} \\
0 & 0 & 1 & a_2 & \cdots & \frac{a_2^{m-3}}{(m-3)!} & \frac{a_2^{m-2}}{(m-2)!} \\
\vdots & \vdots & \vdots & \vdots & \ddots & \vdots & \vdots \\
0 & 0 & 0 & 0 & \cdots & 1 & a_{m-1} \\
1 & x & \frac{x^2}{2!} & \frac{x^3}{3!} & \cdots & \frac{x^{m-1}}{(m-1)!} & \frac{x^m}{m!}
\end{matrix} \right|.
\end{equation}
\item Shift invariance.
\begin{equation}
g_m(x+y; a_0+y, a_1+y, \dots, a_{m-1}+y)=g_m(x; a_0, a_1, \dots, a_{m-1}).
\end{equation}
\end{enumerate}
We note that the number of $\uR$-parking functions of length $m$ is $|\PF(\uR)|=(-1)^m g_m(0; u_1, \dots, u_m)$. For a full discussion of the connection between Gon\v{c}arov polynomials and $\uR$-parking functions, we refer to Kung and Yan \cite{KuYa}.

\begin{theorem}\label{component}
Take $1\leq l \leq m$ any integer. Let $\wR=(w_1, \dots, w_l) \in [u_m]^l$ be in non-decreasing order. The number of parking functions $\piR\in \PF(\uR)$ with $\pi_1=w_1, \dots, \pi_l=w_l$ is
\begin{equation}
(-1)^{m-l}\sum_{\sR \in S_l(\wR)} \binom{m-l}{\sR} \prod_{i=1}^{l+1} g_{s_i}(u_{s_1+\cdots+s_{i-1}+i-1}; u_{s_1+\cdots+s_{i-1}+i}, \dots, u_{s_1+\cdots+s_i+i-1}),
\end{equation}
where $s_0=u_0=0$, $g_{s_i}(\cdot)$ are the Gon\v{c}arov polynomials, and
\begin{equation}
S_l(\wR)=\left\{\sR=(s_1, \dots, s_{l+1}) \in \mathbb{N}^{l+1} \left|\right. \substack{u_{s_1+\cdots +s_i+i} \geq w_i \hspace{.1cm} \forall i\in [l] \\ s_1+\cdots+s_{l+1}=m-l}\right\}.
\end{equation}
Note that this quantity stays constant if all $w_i \leq u_i$ and decreases as each $w_i$ increases past $u_i$ as there are fewer resulting summands.
\end{theorem}

\begin{proof}
If $\pi_i=w_i$ for $1\leq i\leq l$, then $A_{\pi_{l+1}, \dots, \pi_m}=[\vR]$ where $v_i=u_{k_i}\geq w_i$ and $k_i \geq k_{i-1}+1$ with $k_0=u_0=0$. Set $k_{l+1}=m+1$. Thus from Theorem \ref{main1}, the number of $\uR$-parking functions with $\pi_1=w_1, \dots, \pi_l=w_l$ is
\begin{align}
&\sum_{\forall i\in [l]: \hspace{.1cm} \substack{k_i \geq k_{i-1}+1 \\ u_{k_i} \geq w_i}} \binom{m-l}{k_1-k_0-1, \dots, k_{l+1}-k_l-1} \prod_{i=1}^{l+1} (-1)^{k_i-k_{i-1}-1} g_{k_i-k_{i-1}-1}(u_{k_{i-1}}; u_{k_{i-1}+1}, \dots, u_{k_i-1}) \notag \\
&=(-1)^{m-l}\sum_{\sR \in S_l(\wR)} \binom{m-l}{\sR} \prod_{i=1}^{l+1} g_{s_i}(u_{s_1+\cdots+s_{i-1}+i-1}; u_{s_1+\cdots+s_{i-1}+i}, \dots, u_{s_1+\cdots+s_i+i-1}), \label{withk}
\end{align}
where $\sR=(k_1-k_0-1, \dots, k_{l+1}-k_l-1)$.
\end{proof}

For the special case $l=0$ and $\wR=()$ (where no parking preferences are specified), we recover the total number of $\uR$-parking functions $|\PF(\uR)|=(-1)^m g_m(0; u_1, \dots, u_m)$. We describe an alternative characterization of this number in the following.

\begin{proposition}\label{un2}
The number of $\uR$-parking functions $|\PF(\uR)|$ satisfies
\begin{equation}
|\PF(\uR)|=\sum_{\sR \in C(m)} \binom{m}{\sR} \prod_{i=1}^{u_m-m+1} (s_i+1)^{s_i-1},
\end{equation}
where $C(m)$ consists of compositions of $m$: $\sR=(s_1, \dots, s_{u_m-m+1}) \models m$ with $\sum_{i=1}^{u_m-m+1} s_i=m$, subject to $s_1+\cdots+s_{u_i-i+1}\geq i$ for all $1\leq i\leq m$.
\end{proposition}

\begin{proof}
For a parking function $\piR\in \PF(\uR)$, there are $u_m-m$ parking spots that are never attempted by any car. Let $j_i(\piR)$ for $i=1, \dots, u_m-m$ represent these spots, so that $0:=j_0<j_1<\cdots<j_{u_m-m}<j_{u_m-m+1}:=u_m+1$. Set $u_0=0$. This separates $\piR$ into $u_m-m+1$ disjoint non-interacting segments (some segments might be empty), with each segment a classical parking function of length $(j_{i}-j_{i-1}-1)$ after translation, consisting of $\pi_k \to \pi_k-j_{i-1}$ for $j_{i-1}<\pi_k<j_i$. Additionally, there are some constraints imposed on the $j_i$'s. Let $\tauR(\piR)$ denote the parking outcome of $\piR$, where the $i$th car parks in spot $\tau_i$ with $1\leq \tau_i\leq u_m$. Since $\uR$ is strictly increasing, the increasing rearrangement $\lambdaR=(\lambda_1, \dots, \lambda_m)$ of $\tauR$ satisfies $\lambda_i \leq u_i$ for all $1\leq i\leq m$. We note that $\lambda_i \leq u_i$ is equivalent to saying that there are at least $i$ taken spots within the first $u_i$ spots, which is further equivalent to having at most $u_i-i$ empty spots within the first $u_i$ spots. This is achieved if and only if $j_{u_i-i+1}>u_i$. Recall that there are $(n+1)^{n-1}$ classical parking functions of length $n$. We have
\begin{align}
&|\PF(\uR)|=\sum_{\forall i\in [m]: \hspace{.1cm} j_{u_i-i+1}>u_i} \prod_{i=1}^{u_m-m+1} \binom{m}{j_1-j_0-1, \dots, j_{u_m-m+1}-j_{u_m-m}-1} (j_{i}-j_{i-1})^{j_{i}-j_{i-1}-2} \notag \\
&=\sum_{\substack{\sR \models m \\ s_1+\cdots+s_{u_i-i+1}\geq i \hspace{.1cm} \forall i\in [m]}} \binom{m}{s_1, \dots, s_{u_m-m+1}} \prod_{i=1}^{u_m-m+1} (s_i+1)^{s_i-1},
\end{align}
where $\sR=(j_1-j_0-1, \dots, j_{u_m-m+1}-j_{u_m-m}-1)$.
\end{proof}

\begin{example}
Take $\uR=(2, 5)$. Then $(\pi_1, \pi_2) \in \PF(\uR)$ satisfies
\begin{align}
&(\pi_1, \pi_2) \in A:= \{(1, 1), (1, 2), (1, 3), (1, 4), (1, 5), (2, 1), \notag \\
&(2, 2), (2, 3), (2, 4), (2, 5), (3, 1), (3, 2), (4, 1), (4, 2), (5, 1), (5, 2)\}.
\end{align}
From Theorem \ref{un2},
\begin{align}
&|A|=\binom{2}{2, 0, 0, 0} 3^1 1^{-1}1^{-1} 1^{-1}+\binom{2}{0, 2, 0, 0} 1^{-1} 3^1 1^{-1} 1^{-1}+\binom{2}{1, 1, 0, 0} 2^0 2^0 1^{-1} 1^{-1}\notag \\
&\hspace{1.5cm}+\binom{2}{1, 0, 1, 0} 2^0 1^{-1} 2^0 1^{-1}+\binom{2}{1, 0, 0, 1} 2^0 1^{-1} 1^{-1} 2^0\notag \\
&\hspace{1.5cm}+\binom{2}{0, 1, 1, 0} 1^{-1} 2^0 2^0 1^{-1}+\binom{2}{0, 1, 0, 1} 1^{-1} 2^0 1^{-1} 2^0=3+3+2+2+2+2+2=16.
\end{align}
\end{example}

\section{Properties of random $(a, b)$-parking functions}
\label{random}
In general, there are no nice closed-form expressions for Gon\v{c}arov polynomials related to $\uR$-parking functions, but such expressions exist for a specific class of $\uR$-parking functions. When the entries of the vector $\uR$ form an arithmetic progression: $u_i=a+(i-1)b$ for some positive integers $a$ and $b$, we get \emph{Abel polynomials}:
\begin{equation}\label{gon}
g_m(x; a, a+b, \dots, a+(m-1)b)=(x-a)(x-a-mb)^{m-1}.
\end{equation}
We call these $\uR$-parking functions $(a, b)$-parking functions, and denote the set of $(a, b)$-parking functions of length $m$ by $\PF(a, b, m)$. Using (\ref{gon}), the number of $(a, b)$-parking functions of length $m$ is $|\PF(a, b, m)|=(-1)^m g_m(0; a, a+b, \dots, a+(m-1)b)=a(a+mb)^{m-1}$.

In this section we use the multi-shuffle construction introduced in Section \ref{pfs} to investigate various properties of a parking function chosen uniformly at random from $\PF(a, b, m)$. As stated in the introduction, taking $b=1$, $c=n/m-1$, and $a=cm+b$, an $(a, b)$-parking function of length $m$ depicts the scenario of parking $m$ cars in $n$ spots sequentially along a one-way street. Therefore among all possible $a$ and $b$, of particular interest to us is when $a=cm+b$ for some $c\geq 0$. We will write our results in terms of coordinates $\pi_1, \dots, \pi_l$ of parking functions, where $1\leq l\leq m$ is any integer. However, the parking coordinates satisfy permutation symmetry, so the statements in this section may be interpreted for any coordinates.

Before proceeding with the calculations, we outline an effective method for generating a random $(a, b)$-parking function of length $m$. The algorithm is suggested by Stanley's generalization \cite{Stanley1} of Pollak's circle argument for parking functions \cite{Pollak}. To select $\pi \in \PF(a, b, m)$ uniformly at random:
\begin{enumerate}
\item Pick an element $\pi \in (\mathbb{Z}/(a+mb)\mathbb{Z})^m$, where the equivalence class representatives are taken in $1, \dots, a+mb$.

\item For $k \in \{0, \dots, a+mb-1\}$, record $k$ if $\pi+k(1, \dots, 1)\in \PF(a, b, m)$ (modulo $a+mb$), where $(1, \dots, 1)$ is a vector of length $m$. There should be exactly $a$ such $k$'s.

\item Pick one $k$ from (2) uniformly at random. Then $\pi+k(1, \dots, 1)$ is an $(a, b)$-parking function of length $m$ taken uniformly at random.
\end{enumerate}

Figure \ref{distribution} shows a histogram of the values of $\pi_1$ based on $100, 000$ random samples of $\PF(cm+b, b, m)$ for $m=100$ and $b=2$. The left plot is for $c=1$ and the right plot is for $c=0$. A closed formula for the distribution of $\pi_1$ as well as its asymptotic approximation will be provided in Section \ref{section-boundary}.

\begin{figure}
\begin{center}\includegraphics[width=6in]{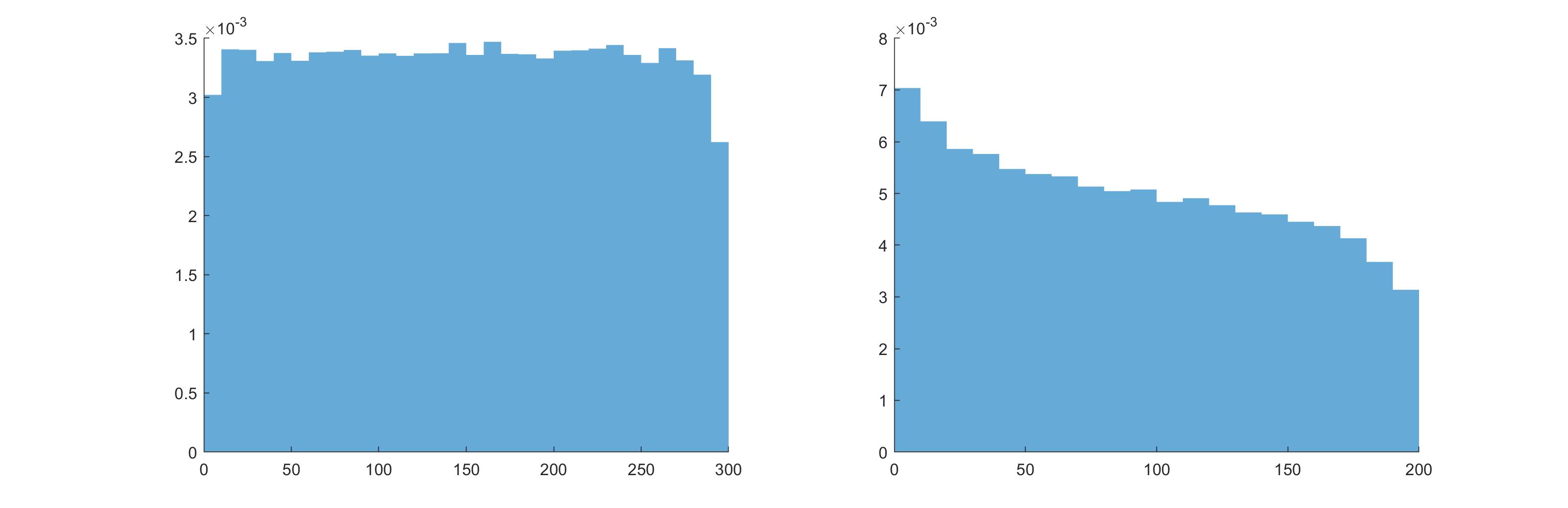}\end{center}
\caption{\label{distribution}
The distribution of $\pi_1$ (the first parking coordinate) in $100,000$ samples of $(cm+b, b)$-parking functions chosen uniformly at random, where $m=100$ and $b=2$. In the left plot $c=1$ and in the right plot $c=0$.}
\end{figure}

\subsection{Mixed moments of multiple coordinates}\label{generic-case}
In this subsection we study the asymptotics of the generic mixed moments of $(a, b)$-parking functions of length $m$ when $a=cm+b$ for some $c>0$ via a tree function approach. Building upon Theorem \ref{component} and Proposition \ref{un2}, we first count the number of $(a, b)$-parking functions of length $m$ where the specified parking preferences of the first $l$ cars are exactly $b$ spots apart. This calculation will be central in deriving an asymptotic formula for the mixed moments.

\begin{proposition}\label{un}
Take $1\leq l \leq m$ any integer. Let $0\leq k\leq m-l$. The number of parking functions $\piR\in \PF(a, b, m)$ with $\pi_1=a+kb, \dots, \pi_l=a+(k+l-1)b$ is
\begin{equation}
a \sum_{s=0}^{m-l-k} \binom{m-l}{s} (a+(m-l-s)b)^{m-s-l-1}b^s l (s+l)^{s-1}.
\end{equation}
\end{proposition}

\begin{proof}
We take $w_i=a+(k+i-1)b$ for $1\leq i\leq l$ in Theorem \ref{component} and extract $s_1$ from the multinomial coefficient $\binom{m-l}{\sR}$:
\begin{align}\label{original}
&(-1)^{m-l} \sum_{s_1=k}^{m-l} \binom{m-l}{s_1} g_{s_1}(0; a, \dots, a+(s_1-1)b)\Big[ \sum_{(s_2, \dots, s_{l+1}) \models m-l-s_1} \binom{m-l-s_1}{s_2, \dots, s_{l+1}} \cdot \notag \\
&\cdot \prod_{i=2}^{l+1} g_{s_i}(a+(s_1+\cdots+s_{i-1}+i-2)b; a+(s_1+\cdots+s_{i-1}+i-1)b, \dots, a+(s_1+\cdots+s_i+i-2)b)\Big] \notag \\
&=a \sum_{s_1=k}^{m-l} \binom{m-l}{s_1} (a+s_1b)^{s_1-1}b^{m-l-s_1} \sum_{(s_2, \dots, s_{l+1}) \models m-l-s_1} \binom{m-l-s_1}{s_2, \dots, s_{l+1}} \prod_{i=2}^{l+1} (s_i+1)^{s_i-1}.
\end{align}
We compare against Proposition \ref{un2} and note that taking $\uR=(l, l+1, \dots, m-1-s_1)$,
\begin{multline}\label{sub}
\sum_{(s_2, \dots, s_{l+1}) \models m-l-s_1} \binom{m-l-s_1}{s_2, \dots, s_{l+1}} \prod_{i=2}^{l+1} (s_i+1)^{s_i-1}=|\PF(\uR)|\\
=(-1)^{m-l-s_1}g_{m-l-s_1}(0; l, l+1, \dots, m-1-s_1)=l(m-s_1)^{m-l-s_1-1}.
\end{multline}
Using (\ref{sub}), (\ref{original}) then becomes
\begin{align}
&a \sum_{s_1=k}^{m-l} \binom{m-l}{s_1} (a+s_1b)^{s_1-1}b^{m-l-s_1} l (m-s_1)^{m-l-s_1-1} \notag \\
&=a \sum_{s=0}^{m-l-k} \binom{m-l}{s} (a+(m-l-s)b)^{m-s-l-1}b^s l (s+l)^{s-1},
\end{align}
where the last equality is a change of variables $s=m-l-s_1$.
\end{proof}

The following technical lemma will also be needed in the derivation of asymptotic mixed moments.

\begin{lemma}\label{technical_lemma}
Take $l \geq 1$ any integer and $m$ large. Take $b \geq 1$ any integer and $a=cm+b$ for some $c>0$. Set $u_i=a+(i-1)b$ for $i \geq 1$ and $u_0=0$. For $1\leq i\leq l$, take $p_i \geq 1$ any integer and $S_i \sim m$ with $S_1<\cdots<S_l$. Let $\sum_{j=u_{s}+1}^{u_{s+1}} j^{p_i}=f_i(s)$ for $s\geq 0$. Then
\begin{equation}\label{left}
\sum_{\substack{\#\{i: \hspace{.05cm} s_i \leq S_k\} \geq k \\ \forall k \in [l]}} \prod_{i=1}^l f_i(s_i)
=\frac{(cm+(S_l+1)b)^{\sum_{i=1}^l p_i+l}}{\prod_{i=1}^l (p_i+1)} \left(1+\frac{1}{(b+c)m}\left(\frac{\sum_{i=1}^l p_i+l}{2}\right)+O\left(\frac{1}{m^2}\right)\right).
\end{equation}
\end{lemma}

\begin{proof}
Notice that the left side of (\ref{left}) may be alternatively computed in stages.

\vspace{.1cm}

Stage 1: We sum up $\prod_{i=1}^l f_i(s_i)$, where the $s_i$'s all range from $0$ to $S_l$.

Stage 2: We subtract the sum of $\prod_{i=1}^l f_i(s_i)$, where the $s_i$'s all range from $S_1+1$ to $S_l$ (so none of the $s_i$'s $\leq S_1$).

Stage 3: We subtract the sum of $\prod_{i=1}^l f_i(s_i)$, where one of the $s_i$'s ranges from $0$ to $S_1$ while the others all range from $S_2+1$ to $S_l$ (so only one of the $s_i$'s $\leq S_2$).

$\vdots$

Stage $l$: We subtract the sum of $\prod_{i=1}^l f_i(s_i)$, where one of the $s_i$'s ranges from $0$ to $S_1$, one ranges from $0$ to $S_2$, $\dots$, one ranges from $0$ to $S_{l-2}$, while the two remaining $s_i$'s both range from $S_{l-1}+1$ to $S_l$ (so only $l-2$ of the $s_i$'s $\leq S_{l-1}$).

\vspace{.1cm}

\noindent For illustration, we perform this alternative procedure when $l=3$.
\begin{align}
&\sum_{s_1=0}^{S_3} \sum_{s_2=0}^{S_3} \sum_{s_3=0}^{S_3} f_1(s_1)f_2(s_2)f_3(s_3)-\sum_{s_1=S_1+1}^{S_3} \sum_{s_2=S_1+1}^{S_3} \sum_{s_3=S_1+1}^{S_3} f_1(s_1)f_2(s_2)f_3(s_3)\notag \\
&-\left(\sum_{s_1=0}^{S_1} f_1(s_1) \sum_{s_2=S_2+1}^{S_3} \sum_{s_3=S_2+1}^{S_3} f_2(s_2) f_3(s_3)+\sum_{s_2=0}^{S_1} f_2(s_2) \sum_{s_1=S_2+1}^{S_3} \sum_{s_3=S_2+1}^{S_3} f_1(s_1) f_3(s_3)\right.\notag \\
&\hspace{5cm} \left.+\sum_{s_3=0}^{S_1} f_3(s_3) \sum_{s_1=S_2+1}^{S_3} \sum_{s_2=S_2+1}^{S_3} f_1(s_1) f_2(s_2)\right).
\end{align}

Since $S_i \sim m$, the sums subtracted in Stages $2$ through $l$ are all of lower order than the sum in Stage $1$. The conclusion then follows from standard asymptotic analysis on the leading order term, which may be estimated by
\begin{equation}
\prod_{i=1}^l \sum_{s_i=0}^{S_l} f_i(s_i)=\prod_{i=1}^l \sum_{j=1}^{a+S_l b} j^{p_i}=\prod_{i=1}^l \frac{(cm+(S_l+1)b)^{p_i+1}}{p_i+1}\left(1+\frac{p_i+1}{2(b+c)m}+O\left(\frac{1}{m^2}\right)\right).
\end{equation}
\end{proof}

We are now ready to establish an asymptotic result for the mixed moments of two coordinates.

\begin{theorem}\label{mean}
Take $p, q\geq 1$ any integer and $m$ large. Take $b \geq 1$ any integer and $a=cm+b$ for some $c>0$. For parking function $\piR$ chosen uniformly at random from $\PF(a, b, m)$, we have
\begin{equation}\label{moment1}
\ER(\pi_1^p)=\frac{((b+c)m)^p}{p+1} \left(1+\frac{1}{c(b+c)m}\left(\frac{c(p+1)}{2}-b^2p\right)+O\left(\frac{1}{m^2}\right)\right),
\end{equation}
and
\begin{equation}\label{moment2}
\ER(\pi_1^p \pi_2^q)=\frac{((b+c)m)^{p+q}}{(p+1)(q+1)} \left(1+\frac{1}{c(b+c)m}\left(\frac{c(p+q+2)}{2}-b^2(p+q)\right)+O\left(\frac{1}{m^2}\right)\right).
\end{equation}
\end{theorem}

\begin{proof}
Set $u_i=a+(i-1)b$ for $1\leq i\leq m$ and $u_0=0$. We break apart the parking preferences of the first two cars into blocks of size $u_i-u_{i-1}$:
\begin{multline}\label{above}
\sum_{j=1}^{u_m} \sum_{k=1}^{u_m} j^p k^q \#\{\piR\in \PF(a, b, m): \pi_1=j, \pi_2=k\}\\
=\sum_{s_1=0}^{m-1} \sum_{j=u_{s_1}+1}^{u_{s_1+1}} \sum_{s_2=0}^{m-1} \sum_{k=u_{s_2}+1}^{u_{s_2+1}} j^p k^q \#\{\piR\in \PF(a, b, m): \pi_1=u_{s_1+1}, \pi_2=u_{s_2+1}\}.
\end{multline}
Let $\sum_{j=u_{s}+1}^{u_{s+1}} j^p=f(s)$ and $\sum_{k=u_{s}+1}^{u_{s+1}} k^q=g(s)$. Then (\ref{above}) is equivalent to
\begin{multline}\label{begin}
\sum_{s_1=0}^{m-2}\sum_{s_2=s_1+1}^{m-1} \#\{\piR\in \PF(a, b, m): \pi_1=u_{s_1+1}, \pi_2=u_{s_2+1}\} (f(s_1)g(s_2)+f(s_2)g(s_1))\\
+\sum_{s_1=0}^{m-2} \#\{\piR\in \PF(a, b, m): \pi_1=u_{s_1+1}, \pi_2=u_{s_1+2}\} f(s_1)g(s_1).
\end{multline}

By Theorem \ref{component}, the first term of (\ref{begin}) is
\begin{align}\label{mid}
&a \sum_{s_1=0}^{m-2} \sum_{s_2=s_1+1}^{m-1} (f(s_1)g(s_2)+f(s_2)g(s_1)) \sum_{t_1=s_1}^{m-2} \sum_{t_2=\max(0, s_2-1-t_1)}^{m-2-t_1} \binom{m-2}{t_1, t_2, m-2-t_1-t_2} \cdot \notag \\
&\hspace{1.5cm} \cdot b^{m-2-t_1} (a+t_1b)^{t_1-1}(t_2+1)^{t_2-1} (m-2-t_1-t_2+1)^{m-2-t_1-t_2-1}\notag \\
&=a \sum_{t_1=0}^{m-2} \sum_{t_2=0}^{m-2-t_1} \binom{m-2}{t_1, t_2, m-2-t_1-t_2} b^{m-2-t_1} (a+t_1b)^{t_1-1}(t_2+1)^{t_2-1} \cdot \notag \\ &\hspace{1.5cm} \cdot (m-2-t_1-t_2+1)^{m-2-t_1-t_2-1} \sum_{s_1=0}^{t_1} \sum_{s_2=s_1+1}^{t_1+t_2+1} (f(s_1)g(s_2)+f(s_2)g(s_1)).
\end{align}
We make a change of variables: $s=t_2$ and $t=m-2-t_1-t_2$. Then (\ref{mid}) becomes
\begin{align}\label{middle}
&a \sum_{s=0}^{m-2} \sum_{t=0}^{m-2-s} \binom{m-2}{s, t, m-2-s-t} b^{s+t} (a+(m-2-s-t)b)^{m-3-s-t} \cdot \notag \\
&\hspace{1.5cm} \cdot (s+1)^{s-1}(t+1)^{t-1} \sum_{s_1=0}^{m-2-s-t} \sum_{s_2=s_1+1}^{m-1-t} (f(s_1)g(s_2)+f(s_2)g(s_1)).
\end{align}
Similarly, by Proposition \ref{un}, the second term of (\ref{begin}) is
\begin{align}\label{simple}
&a \sum_{s_1=0}^{m-2} f(s_1)g(s_1) \sum_{t_1=s_1}^{m-2} \sum_{t_2=0}^{m-2-t_1} \binom{m-2}{t_1, t_2, m-2-t_1-t_2} \cdot \notag \\
&\hspace{1.5cm} \cdot b^{m-2-t_1} (a+t_1b)^{t_1-1}(t_2+1)^{t_2-1} (m-2-t_1-t_2+1)^{m-2-t_1-t_2-1}\notag \\
&=a \sum_{t_1=0}^{m-2} \sum_{t_2=0}^{m-2-t_1} \binom{m-2}{t_1, t_2, m-2-t_1-t_2} b^{m-2-t_1} (a+t_1b)^{t_1-1}(t_2+1)^{t_2-1} \cdot \notag \\ &\hspace{1.5cm} \cdot (m-2-t_1-t_2+1)^{m-2-t_1-t_2-1} \sum_{s_1=0}^{t_1} f(s_1)g(s_1).
\end{align}
We make a change of variables: $s=t_2$ and $t=m-2-t_1-t_2$. Then (\ref{simple}) becomes
\begin{align}\label{middle2}
&a \sum_{s=0}^{m-2} \sum_{t=0}^{m-2-s} \binom{m-2}{s, t, m-2-s-t} b^{s+t} (a+(m-2-s-t)b)^{m-3-s-t} \cdot \notag \\
&\hspace{1.5cm} \cdot (s+1)^{s-1}(t+1)^{t-1} \sum_{s_1=0}^{m-2-s-t} f(s_1)g(s_1).
\end{align}

Using Lemma \ref{technical_lemma}, for $p, q \geq 1$, (\ref{middle})+(\ref{middle2}) is asymptotically
\begin{align}
&\label{before}a \sum_{s=0}^{m-2} \sum_{t=0}^{m-2-s} \binom{m-2}{s, t, m-2-s-t} b^{s+t} (a+(m-2-s-t)b)^{m-3-s-t} \cdot \notag \\
&\hspace{1.5cm} \cdot (s+1)^{s-1}(t+1)^{t-1} \frac{(cm+(m-t)b)^{p+q+2}}{(p+1)(q+1)} \left(1+\frac{1}{(b+c)m}\left(\frac{p+q+2}{2}\right)+O\left(m^{-2}\right)\right) \notag \\
&=\frac{cm+b}{(p+1)(q+1)}((b+c)m)^{m+p+q-1} \sum_{s=0}^{m-2} \sum_{t=0}^{m-2-s} \frac{1}{s! t!} \left(\frac{b}{b+c}\right)^{s+t} e^{-\frac{b}{b+c}(s+t+1)} (s+1)^{s-1}(t+1)^{t-1} \cdot \notag \\
&\cdot \left(1-\frac{(s+t)(s+t+3)}{2m}+\frac{b}{b+c}\frac{(s+t+1)(s+t+3)}{m}-\frac{b}{b+c}\frac{t(p+q+2)}{m}\right.\notag \\
& \hspace{6cm}\left.-\frac{b^2}{(b+c)^2}\frac{(s+t+1)^2}{2m}+\frac{p+q+2}{2(b+c)m}+O(m^{-2})\right).
\end{align}
The tree function $F(z) = \sum_{s=0}^\infty \frac{z^s}{s!}(s+1)^{s-1}$ is related to the Lambert $W$ function via $F(z)=-W(-z)/z$, and
satisfies $F(xe^{-x}) = e^x$. By the chain rule its first and second derivatives therefore satisfy
\begin{align}
F'(xe^{-x}) = \frac{e^{2x}}{1-x}, \hspace{1cm} F''(xe^{-x}) = \frac{3-2x}{(1-x)^3}e^{3x}.
\end{align}
We recognize that (\ref{before}) is in the form of a Cauchy product, and converges to
\begin{align}
&\frac{cm+b}{(p+1)(q+1)} ((b+c)m)^{m+p+q-1} e^{-\frac{b}{b+c}} \sum_{s=0}^{\infty} \sum_{t=0}^{\infty}\frac{1}{s! t!} \left(\frac{b}{b+c}e^{-\frac{b}{b+c}}\right)^{s+t} (s+1)^{s-1} (t+1)^{t-1} \cdot \notag \\
& \hspace{2cm} \cdot \left(1+\frac{1}{m} (A+Bs+Ct+Ds^2+Et^2+Fst)+O(m^{-2})\right),
\end{align}
where
\begin{equation*}
A=-\frac{b^2}{2(b+c)^2}+3\frac{b}{b+c}+\frac{p+q+2}{2(b+c)}, \hspace{.2cm} B+C=-2\frac{b^2}{(b+c)^2}+(6-p-q)\frac{b}{b+c}-3,
\end{equation*}
\begin{equation}
D+E=-\frac{b^2}{(b+c)^2}+2\frac{b}{b+c}-1, \hspace{.2cm} F=-\frac{b^2}{(b+c)^2}+2\frac{b}{b+c}-1.
\end{equation}
Using $F(z)$ this can be written as (with $z=\frac{b}{b+c}e^{-\frac{b}{b+c}}$):
\begin{align}
&\frac{cm+b}{(p+1)(q+1)} ((b+c)m)^{m+p+q-1} \cdot \notag \\
&\cdot \left[F(z)+\frac1m\left(AF(z)+(B+C)zF'(z)+(D+E)(z^2F''(z)+zF'(z))+Fz^2 F'(z) \frac{F'(z)}{F(z)}\right)+O(m^{-2})\right].
\end{align}
Dividing by $|\PF(a, b, m)|=(cm+b)((b+c)m+b)^{m-1}$ and simplifying we get
\begin{equation}\label{contri2}
\frac{((b+c)m)^{p+q}}{(p+1)(q+1)} \left(1+\frac{1}{c(b+c)m}\left(\frac{c(p+q+2)}{2}-b^2(p+q)\right)+O\left(\frac{1}{m^2}\right)\right)
\end{equation}
for the generic $(p, q)$-th mixed moment.

For the special case $p\geq 1$ and $q=0$, a similar asymptotic calculation gives the $p$-th moment as
\begin{equation}\label{contri2-special}
\frac{((b+c)m)^p}{p+1} \left(1+\frac{1}{c(b+c)m}\left(\frac{c(p+1)}{2}-b^2p\right)+O\left(\frac{1}{m^2}\right)\right).
\end{equation}
\end{proof}

\begin{proposition}\label{two-coordinates}
Take $m$ large. Take $b \geq 1$ any integer and $a=cm+b$ for some $c>0$. For parking function $\piR$ chosen uniformly at random from $\PF(a, b, m)$, we have
\begin{equation}
\Var(\pi_1) \sim \frac{((b+c)m)^2}{12}-\frac{b^2(b+c)m}{6c}, \hspace{1cm} \Cov(\pi_1, \pi_2) \sim -\frac{b^2(b+c)^2}{4c^2}.
\end{equation}
\end{proposition}

\begin{proof}
Set $x=\frac{b}{b+c}$. For $p=q=1$, performing asymptotic expansion as in the proof of Theorem \ref{mean} but keeping more lower order terms, we have
\begin{align*}
\sum_{j=1}^{a+(m-1)b} \sum_{k=1}^{a+(m-1)b} jk \#\{\piR\in \PF(a, b, m): \pi_1=j, \pi_2=k\}
\end{align*}
converges to
\begin{align*}
&\frac{cm+b}{4} ((b+c)m)^{m+1}e^{-x} \sum_{s=0}^{\infty} \sum_{t=0}^{\infty} \frac{(xe^{-x})^{s+t}}{s! t!} (s+1)^{s-1} (t+1)^{t-1} \cdot \notag \\
&\cdot \left(1+\left(A_1+A_2s+A_3t+A_4s^2+A_5t^2+A_6st\right)\frac{1}{m}+\left(B_1+B_2s+B_3t+B_4s^2+B_5t^2+B_6st+ \right.\right. \notag \\
&\left.+B_7s^2t+B_8st^2+B_9s^3+B_{10}t^3+B_{11}s^3t+B_{12}st^3+B_{13}s^2t^2+B_{14}s^4+B_{15}t^4)\frac{1}{m^2} +O(m^{-3}) \right),
\end{align*}
where
\begin{equation*}
A_1=3x+\frac{2x}{b}-\frac{x^2}{2}, \hspace{.5cm} A_2+A_3=-3+4x-2x^2,
\end{equation*}
\begin{equation}
A_4+A_5=-1+2x-x^2, \hspace{.5cm} A_6=-1+2x-x^2.
\end{equation}
\vspace{.05cm}
\begin{equation*}
B_1=2x^2+\frac{x^2}{b^2}+\frac{6x^2}{b}-\frac{11x^3}{6}-\frac{x^3}{b}+\frac{x^4}{8}, \hspace{.5cm} B_2+B_3=-\frac{13}{6}-9x-\frac{6x}{b}+\frac{25x^2}{2}+\frac{10x^2}{b}-10x^3-\frac{4x^3}{b}+x^4 ,
\end{equation*}
\begin{equation*}
B_4+B_5=\frac{3}{4}-9x-\frac{2x}{b}+\frac{33x^2}{2}+\frac{4x^2}{b}-10x^3-\frac{2x^3}{b}+\frac{3x^4}{2},
\end{equation*}
\begin{equation*}
B_6=\frac{3}{4}-9x-\frac{2x}{b}+\frac{29x^2}{2}+\frac{4x^2}{b}-10x^3-\frac{2x^3}{b}+\frac{3x^4}{2},
\end{equation*}
\begin{equation*}
B_7+B_8=\frac{7}{2}-15x+\frac{45x^2}{2}-14x^3+3x^4, \hspace{.5cm} B_9+B_{10}=\frac{7}{6}-5x+\frac{15x^2}{2}-\frac{14x^3}{3}+x^4,
\end{equation*}
\begin{equation*}
B_{11}+B_{12}=1-4x+6x^2-4x^3+x^4, \hspace{.5cm} B_{13}=\frac{3}{4}-3x+\frac{9x^2}{2}-3x^3+\frac{3x^4}{4},
\end{equation*}
\begin{equation}
B_{14}+B_{15}=\frac{1}{4}-x+\frac{3x^2}{2}-x^3+\frac{x^4}{4}.
\end{equation}
A more involved application of the tree function method then yields
\begin{equation}\label{last2}
\ER(\pi_1 \pi_2) \sim \frac{((b+c)m)^2}{4}+\frac{(b+c)m}{2c}(c-b^2)+\frac{2b^5+8b^4c+6b^3c^2-2b^2c^2+b^2c^3+c^3}{4c^3}.
\end{equation}
The same approach also yields
\begin{equation}\label{last1}
\ER(\pi_1) \sim \frac{(b+c)m}{2}+\frac{c-b^2}{2c}+\frac{b^2(b^2+3bc+c^2)}{2c^3m}.
\end{equation}
The claimed asymptotics are then immediate.
\end{proof}

Extending the asymptotic expansion approach in the proof of Theorem \ref{mean}, we have the following more general result.

\begin{theorem}\label{general-mean}
Take $l\geq 1$ any integer and $m$ large. Take $b \geq 1$ any integer and $a=cm+b$ for some $c>0$. For $1\leq i\leq l$, take $p_i \geq 1$ any integer. For parking function $\piR$ chosen uniformly at random from $\PF(a, b, m)$, we have
\begin{equation}\label{general-moment}
\ER(\prod_{i=1}^l \pi_i^{p_i})=\frac{((b+c)m)^{\sum_{i=1}^l p_i}}{\prod_{i=1}^l (p_i+1)} \left(1+\frac{1}{c(b+c)m}\left(\frac{c\left(\sum_{i=1}^l p_i+l\right)}{2}-b^2\sum_{i=1}^l p_i\right)+O\left(\frac{1}{m^2}\right)\right).
\end{equation}
\end{theorem}

\begin{proof}
We will not include all technical details, but as in the $l=1,2$ case, the key idea in the generic situation is still to break apart the parking preferences of the first $l$ cars into blocks. Then as in the proof of Theorem \ref{mean}, using Theorem \ref{component} and Proposition \ref{un} and interchanging the order of summation, we have
\begin{align}\label{general-asy}
&\sum \left(\prod_{i=1}^l \pi_i^{p_i}\right) \{\piR\in \PF(a, b, m): \pi_i \text{ specified } \forall i \in [l]\} \notag \\
&=a \sum_{s_1=0}^{m-l} \cdots \sum_{s_l=0}^{m-l-s_1-\cdots-s_{l-1}} \binom{m-l}{s_1, \dots, s_l, m-l-s_1-\cdots-s_l} b^{s_1+\cdots+s_l} \cdot \notag \\
&\cdot (a+(m-l-s_1-\cdots-s_l)b)^{m-l-1-s_1-\cdots-s_l} \prod_{i=1}^l (s_i+1)^{s_i-1} \left[\sum_{\substack{\#\{i: \hspace{.05cm} t_i \leq T_k\} \geq k \\ \forall k \in [l]}} \prod_{i=1}^l f_i(t_i)\right],
\end{align}
where $T_k=m-l+k-1-\sum_{j=k}^l s_j$ for $1\leq k \leq l$. Set $x=\frac{b}{b+c}$. By Lemma \ref{technical_lemma}, for $p_i\geq 1$, (\ref{general-asy}) is asymptotically
\begin{align}
&\frac{cm+b}{\prod_{i=1}^l (p_i+1)}((b+c)m)^{m-1+\sum_{i=1}^l p_i} \sum_{s_1=0}^{m-l} \cdots \sum_{s_l=0}^{m-l-s_1-\cdots-s_{l-1}} \frac{x^{\sum_{i=1}^l s_i}}{\prod_{i=1}^l s_i!} e^{-x\left(l-1+\sum_{i=1}^l s_i\right)} \prod_{i=1}^l (s_i+1)^{s_i-1} \cdot \notag \\
&\cdot \left(1-\frac{\left(\sum_{i=1}^l s_i\right)\left(2l-1+\sum_{i=1}^l s_i\right)}{2m}+\frac{x\left(l-1+\sum_{i=1}^l s_i\right)\left(l+1+\sum_{i=1}^l s_i\right)}{m}-\frac{xs_l\left(\sum_{i=1}^l p_i+l\right)}{m}\right.\notag \\
& \hspace{1.5cm}\left.-\frac{x^2\left(l-1+\sum_{i=1}^l s_i\right)^2}{2m}+\frac{x\left(\sum_{i=1}^l p_i+l\right)}{2bm}+O\left(m^{-2}\right)\right). \label{generic}
\end{align}

Denote by $F(z) = \sum_{s=0}^\infty \frac{z^s}{s!}(s+1)^{s-1}$ with $z=xe^{-x}$. An application of the tree function method shows that (\ref{generic}) converges to
\begin{align}
&\frac{cm+b}{\prod_{i=1}^l (p_i+1)} ((b+c)m)^{m-1+\sum_{i=1}^l p_i} \cdot \notag \\
&\cdot \Bigg[F(z)+\frac1m\left(AF(z)+\left(Bl-x\sum_{i=1}^l p_i-xl\right)zF'(z)+Cl(z^2F''(z)+zF'(z))+D\binom{l}{2}z^2 F'(z) \frac{F'(z)}{F(z)}\right)\notag \\
&\hspace{6cm}+O(m^{-2})\Bigg],
\end{align}
where
\begin{equation*}
A=-\frac{x^2(l-1)^2}{2}+x(l^2-1)+\frac{x\left(\sum_{i=1}^l p_i+l\right)}{2b}, \hspace{.5cm} B=-x^2(l-1)+2xl-\frac{2l-1}{2},
\end{equation*}
\begin{equation}
C=-\frac{x^2}{2}+x-\frac{1}{2}, \hspace{.5cm} D=-x^2+2x-1.
\end{equation}
Dividing by $|\PF(a, b, m)|=(cm+b)((b+c)m+b)^{m-1}$ and simplifying we get
\begin{equation}
\frac{((b+c)m)^{\sum_{i=1}^l p_i}}{\prod_{i=1}^l (p_i+1)} \left(1+\frac{1}{c(b+c)m}\left(\frac{c\left(\sum_{i=1}^l p_i+l\right)}{2}-b^2\sum_{i=1}^l p_i\right)+O\left(\frac{1}{m^2}\right)\right)
\end{equation}
for the generic mixed moment.
\end{proof}

\subsection{The special situation $c=0$}\label{special}
In this subsection we study the asymptotics of the generic mixed moments of $(b, b)$-parking functions of length $m$ via Abel's multinomial theorem. Indeed, the asymptotic moment calculations in Section \ref{generic-case} could as well be approached via Abel's multinomial theorem. Unlike the tree function method which fails for the case $c=0$ due to divergence, Abel's multinomial theorem applies broadly for $c\geq 0$. However calculation-wise it is in general more cumbersome to apply Abel's multinomial theorem as compared with the tree function method, so we only use this alternative approach when $c=0$ and so $a=b$.

\begin{theorem}[Abel's multinomial theorem, derived from Pitman \cite{Pitman} and Riordan \cite{Riordan}]\label{Abel}
Let
\begin{equation}\label{b}
A_n(x_1, \dots, x_m; p_1, \dots, p_m)=\sum_{\sR \models n} \binom{n}{\sR} \prod_{j=1}^m (x_j+s_j)^{s_j+p_j},
\end{equation}
where $\sR=(s_1, \dots, s_m)$ and $\sum_{i=1}^m s_i=n$.
Then
\begin{multline}\label{b1}
A_n(x_1, \dots, x_i, \dots, x_j, \dots, x_m; p_1, \dots, p_i, \dots, p_j, \dots, p_m)\\=A_n(x_1, \dots, x_j, \dots, x_i, \dots, x_m; p_1, \dots, p_j, \dots, p_i, \dots, p_m).
\end{multline}

\begin{multline}\label{b2}
A_n(x_1, \dots, x_m; p_1, \dots, p_m)\\=\sum_{i=1}^m A_{n-1}(x_1, \dots, x_{i-1}, x_i+1, x_{i+1}, \dots, x_m; p_1, \dots, p_{i-1}, p_i+1, p_{i+1}, \dots, p_m).
\end{multline}

\begin{equation}\label{b3}
A_n(x_1, \dots, x_m; p_1, \dots, p_m)=\sum_{s=0}^{n} \binom{n}{s}s!(x_1+s)A_{n-s}(x_1+s, x_2, \dots, x_m; p_1-1, p_2, \dots, p_m).
\end{equation}
Moreover, the following special instances hold via the basic recurrences listed above:
\begin{equation}\label{1}
A_n(x_1, \dots, x_m; -1, \dots, -1)=(x_1\cdots x_m)^{-1}(x_1+\cdots+x_m)(x_1+\cdots+x_m+n)^{n-1}.
\end{equation}

\begin{equation}\label{2}
A_n(x_1, \dots, x_m; -1, \dots, -1, 0)=(x_1\cdots x_m)^{-1}x_m(x_1+\cdots+x_m+n)^{n}.
\end{equation}
\end{theorem}

Take $l\geq 1$ any integer and $m$ large. Take $b \geq 1$ any integer and $a=cm+b$ for some $c \geq 0$. For $1\leq i\leq l$, take $p_i \geq 1$ any integer. In computing $\ER(\prod_{i=1}^l \pi_i^{p_i})$, we recognize from Lemma \ref{technical_lemma} that (\ref{general-asy}) is asymptotically
\begin{align}\label{leading}
&\frac{(cm+b)b^{m+\sum_{i=1}^l p_i-1}}{\prod_{i=1}^l (p_i+1)} \sum_{s_1=0}^{m-l} \cdots \sum_{s_l=0}^{m-l-s_1-\cdots-s_{l-1}} \binom{m-l}{s_1, \dots, s_l, m-l-s_1-\cdots-s_l} \cdot \notag \\
&\cdot (m-l-s_1-\cdots-s_l+1+\frac{cm}{b})^{m-l-1-s_1-\cdots-s_l} \prod_{i=1}^l (s_i+1)^{s_i-1} \cdot \notag \\
&\cdot (m-s_l+\frac{cm}{b})^{\sum_{i=1}^l p_i+l} \left(1+\frac{1}{(b+c)m-s_lb}\left(\frac{\sum_{i=1}^l p_i+l}{2}\right)+O\left(m^{-2}\right)\right).
\end{align}
Using Abel's multinomials, the leading order terms in (\ref{leading}) may be represented as
\begin{align}
&\frac{(cm+b)b^{m+\sum_{i=1}^l p_i-1}}{\prod_{i=1}^l (p_i+1)} \left(A_{m-l}(1+\frac{cm}{b}, \underbracket[0.5pt]{1, \dots, 1}_{l \hspace{.1cm} \text{1's}}; \sum_{i=1}^l p_i+l-1, \underbracket[0.5pt]{-1, \dots, -1}_{l \hspace{.1cm} \text{-1's}})\right.\notag \\
&+(l-1)\left(\sum_{i=1}^l p_i+l\right) A_{m-l}(1+\frac{cm}{b}, \underbracket[0.5pt]{1, \dots, 1}_{l \hspace{.1cm} \text{1's}}; \sum_{i=1}^l p_i+l-2, 0, \underbracket[0.5pt]{-1, \dots, -1}_{l-1 \hspace{.1cm} \text{-1's}})\notag \\
&\left.+\frac{1}{2b}\left(\sum_{i=1}^l p_i+l\right) A_{m-l}(1+\frac{cm}{b}, \underbracket[0.5pt]{1, \dots, 1}_{l \hspace{.1cm} \text{1's}}; \sum_{i=1}^l p_i+l-2, \underbracket[0.5pt]{-1, \dots, -1}_{l \hspace{.1cm} \text{-1's}})\right).
\end{align}
This is a general formula that works for any $a$, $b$, and $l$. When $c=0$ and so $a=b$, taking $l=1, 2$, we have
\begin{equation}
\ER(\pi_1) \sim b\left(\frac{m}{2}-\frac{\sqrt{2\pi}}{4}m^{1/2}+\frac{7}{6}\right)+\frac{1}{2}.
\end{equation}
\begin{equation}
\ER(\pi_1^2) \sim b^2\left(\frac{m^2}{3}-\frac{\sqrt{2\pi}}{4}m^{3/2}+\frac{4}{3}m\right)+\frac{b}{2}m.
\end{equation}
\begin{equation}
\ER(\pi_1 \pi_2) \sim b^2\left(\frac{m^2}{4}-\frac{\sqrt{2\pi}}{4}m^{3/2}+\frac{3}{2}m\right)+\frac{b}{2}m.
\end{equation}
These asymptotic results are in sharp contrast with the case $a=cm+b$ where $c>0$. As $c \rightarrow 0$, the correction terms in (\ref{last2}) (\ref{last1}) blow up, contributing to the different asymptotic orders between the generic situation $a \gtrsim b$ (corresponding to $c>0$) and the special situation $a=b$ (corresponding to $c=0$). Paralleling Proposition \ref{two-coordinates}, the following asymptotics are immediate.

\begin{proposition}\label{two-coordinates-special}
Take $m$ large. Take $b \geq 1$ any integer. For parking function $\piR$ chosen uniformly at random from $\PF(b, b, m)$, we have
\begin{equation}
\Var(\pi_1) \sim \frac{b^2}{12}m^2+\frac{b^2(4-3\pi)}{24}m, \hspace{1cm}
\Cov(\pi_1, \pi_2) \sim \frac{b^2(8-3\pi)}{24}m.
\end{equation}
\end{proposition}

For $\piR=(\pi_1, \dots, \pi_m) \in \PF(a, b, m)$, the $(a, b)$-displacement $\disp^{(a, b)}(\piR)$ is defined as
\begin{equation}
\disp^{(a, b)}(\piR)=b\binom{m}{2}+am-(\pi_1+\cdots+\pi_m).
\end{equation}
Figure \ref{distribution2} shows a histogram of the displacement based on $100, 000$ random samples of $\PF(cm+b, b, m)$
for $m=100$ and $b=2$. The left plot ($c=1$) approximates a normal distribution and the right plot ($c=0$) approximates an Airy distribution. The displacement definition is in connection with the displacement enumerator of $(a, b)$-parking functions. Note that the set of $(a, b)$-parking functions of length $m$ is in bijection with the set of length-$a$ sequences of rooted $b$-forests on $m$ vertices, and there are related formulations for the $(a, b)$-inversion and inversion enumerator of length-$a$ sequences of rooted $b$-forests. See Yan \cite{Yan1} for more details.

\begin{theorem}\label{adapted}
Take $m$ large. Take $b \geq 1$ any integer and $a=cm+b$ for some $c>0$. For parking function $\piR$ chosen uniformly at random from $\PF(a, b, m)$, we have
\begin{equation}
\ER(\disp^{(a, b)}(\piR))\sim \frac{cm^2}{2}+\left(\frac{b^2}{2c}+\frac{b}{2}-\frac{1}{2}\right)m, \hspace{1cm} \Var(\disp^{(a, b)}(\piR))\sim \frac{(b+c)^2}{12} m^3.
\end{equation}
Contrarily, when $c=0$ and so $a=b$,
\begin{equation}
\ER(\disp^{(a, b)}(\piR))\sim b\frac{\sqrt{2\pi}}{4}m^{3/2}-\left(\frac{2b}{3}+\frac{1}{2}\right)m, \hspace{1cm} \Var(\disp^{(a, b)}(\piR))\sim \frac{b^2(10-3\pi)}{24} m^3.
\end{equation}
\end{theorem}

\begin{proof}
By permutation symmetry,
\begin{equation}
\ER(\disp^{(a, b)}(\piR))=b\binom{m}{2}+am-\ER(\pi_1).
\end{equation}
\begin{equation}
\Var(\disp^{(a, b)}(\piR))=m\Var(\pi_1)+m(m-1)\Cov(\pi_1, \pi_2).
\end{equation}
We then combine the results from Propositions \ref{two-coordinates} and \ref{two-coordinates-special}.
\end{proof}

\begin{figure}
\begin{center}\includegraphics[width=6in]{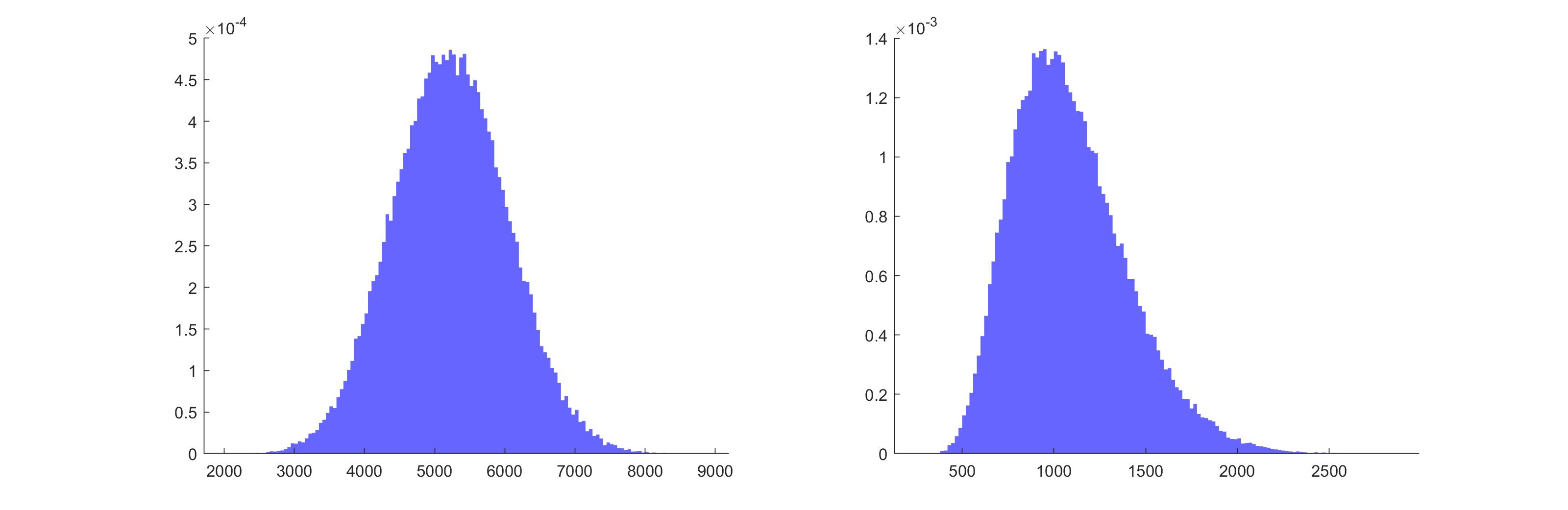}\end{center}
\caption{\label{distribution2}
The distribution of displacement in $100,000$ samples of $(cm+b, b)$-parking functions chosen uniformly at random, where $m=100$ and $b=2$. In the left plot $c=1$ and in the right plot $c=0$.}
\end{figure}

\subsection{Boundary behavior of a single coordinate}\label{section-boundary}
In this subsection we examine the boundary behavior of a single coordinate of $(a, b)$-parking functions of length $m$ when $a=cm+b$ for some $c \geq 0$. As in the case of the distribution of multiple coordinates, the asymptotic tendency in the generic situation $c>0$ and the special situation $c=0$ are strikingly different. Calculational techniques (tree function, Abel's multinomial theorem) employed in Sections \ref{generic-case} and \ref{special} will be used in our investigation, with details omitted.
\begin{proposition}\label{1st}
Let $1\leq j\leq a+(m-1)b$. The number of parking functions $\piR \in \PF(a, b, m)$ with $\pi_1=j$ is
\begin{equation}
ab \sum_{s: \hspace{.05cm} a+sb \geq j} \binom{m-1}{s} (a+sb)^{s-1} ((m-s)b)^{m-2-s}.
\end{equation}
Note that this quantity stays constant for $j \leq a$ and decreases as $j$ increases past $a$ as there are fewer resulting summands.
\end{proposition}

\begin{proof}
From Theorem \ref{component}, the number of $(a, b)$-parking functions of length $m$ with $\pi_1=j$ is
\begin{align}
&(-1)^{m-1} \sum_{s: \hspace{.05cm} a+sb \geq j} \binom{m-1}{s} g_s\left(0; a, \dots, a+(s-1)b\right) g_{m-1-s}\left(a+sb; a+(s+1)b, \dots, a+(m-1)b\right) \notag \\
&=(-1)^{m-1} \sum_{s: \hspace{.05cm} a+sb \geq j} \binom{m-1}{s} (-a)(-a-sb)^{s-1} (-b)(-b-(m-1-s)b)^{m-2-s} \notag \\
&=ab \sum_{s: \hspace{.05cm} a+sb \geq j} \binom{m-1}{s} (a+sb)^{s-1} ((m-s)b)^{m-2-s},
\end{align}
where we used (\ref{gon}) in the first equality.
\end{proof}

Recall from Proposition \ref{nonempty} that if $A_{\pi_2, \dots, \pi_m}=[v]$ is non-empty, then $v=a+sb$ for some $0\leq s\leq m-1$. Let $X$ be a random variable satisfying the Borel distribution with parameter $\mu$ ($0\leq \mu\leq 1$), that is, with pmf given by, for $j=1, 2, \dots$,
\begin{equation}
\PR_\mu(X=j)=\frac{e^{-\mu j}(\mu j)^{j-1}}{j!}.
\end{equation}
Denote by $\QR_\mu(j)=\PR_\mu(X\geq j)$. We refer to Stanley \cite{Stanley} for some nice properties of this discrete distribution.

\begin{corollary}\label{boundary2}
Fix $j$ and take $m$ large relative to $j$. Take $b\geq 1$ any integer and $a=cm+b$ for some $c \geq 0$. For parking function $\piR$ chosen uniformly at random from $\PF(a, b, m)$, we have
\begin{equation}
\PR(\pi_1=a+(m-1)b-j)\sim \frac{1-\QR_{b/(b+c)}(\lfloor j/b \rfloor+2)}{(b+c)m},
\end{equation}
where $\QR_{b/(b+c)}(l) = \PR_{b/(b+c)}(X\geq l)$ is the tail distribution function of Borel-$b/(b+c)$ and $\lfloor x \rfloor$ denotes the largest integer $\leq x$.
\end{corollary}

\begin{proof}
If $\pi_1=a+(m-1)b-j$, then $A_{\pi_2, \dots, \pi_m}=[a+sb]$ for some $s$ so that $sb\geq (m-1)b-j$. From Proposition \ref{1st}, this implies that
\begin{align}
&\PR(\pi_1=a+(m-1)b-j)=\PR(\pi_1=a+(m-1)b)+\sum_{s=m-1-\lfloor j/b\rfloor}^{m-2}\PR(A_{\pi_2, \dots, \pi_m}=[a+sb]) \notag \\
&=\PR(\pi_1=a+(m-1)b-j)+\sum_{s=m-1-\lfloor j/b\rfloor}^{m-2} \frac{b \binom{m-1}{m-1-s} (a+sb)^{s-1} ((m-s)b)^{m-2-s}}{(a+mb)^{m-1}} \notag \\
& \sim \PR(\pi_1=a+(m-1)b-j)+\frac{1}{(b+c)m}\left(\QR_{b/(b+c)}(2)-\QR_{b/(b+c)}(\lfloor j/b \rfloor+2)\right),
\end{align}
where we use that $j$ (and hence $m-1-s$) is small relative to $m$.

Hence we only need to check the boundary case:
\begin{equation}
\PR(\pi_1=a+(m-1)b)=\frac{(a+(m-1)b)^{m-2}}{(a+mb)^{m-1}}\sim \frac{1}{(b+c)me^{b/(b+c)}}=\frac{\PR_{b/(b+c)}(X=1)}{(b+c)m}=\frac{1-\QR_{b/(b+c)}(2)}{(b+c)m}.
\end{equation}
\end{proof}

\begin{corollary}\label{boundary1}
Fix $j$ and take $m$ large relative to $j$. Take $b\geq 1$ any integer and $a=cm+b$ for some $c>0$. For parking function $\piR$ chosen uniformly at random from $\PF(a, b, m)$, we have
\begin{equation}
\PR(\pi_1=1)=\cdots=\PR(\pi_1=a) \sim \frac{1}{(b+c)m},
\end{equation}
and
\begin{equation}
\PR(\pi_1=a+j)\sim e^{\frac{cm}{be}-\frac{b}{b+c}} \left(\frac{b}{b+c}\right)^{m-1} \frac{1}{cm^2}\left(\PR(Y\geq \lceil j/b\rceil)-1\right)+\frac{1}{(b+c)m},
\end{equation}
where $Y$ is a Poisson$((cm)/(be))$ random variable and $\lceil x \rceil$ denotes the smallest integer $\geq x$.
\end{corollary}

\begin{proof}
If $\pi_1=a+j$, then $A_{\pi_2, \dots, \pi_m}=[a+sb]$ for some $s$ so that $sb\geq j$. From Proposition \ref{1st}, this implies that
\begin{align}
&\PR(\pi_1=a+j)=\PR(\pi_1=a)-\sum_{s=0}^{\lceil j/b\rceil-1}\PR(A_{\pi_2, \dots, \pi_m}=[a+sb]) \notag \\
&=\PR(\pi_1=a)-\sum_{s=0}^{\lceil j/b\rceil-1} \frac{b \binom{m-1}{s} (a+sb)^{s-1} ((m-s)b)^{m-2-s}}{(a+mb)^{m-1}} \notag \\
& \sim \PR(\pi_1=a)-e^{\frac{cm}{be}-\frac{b}{b+c}} \left(\frac{b}{b+c}\right)^{m-1} \frac{1}{cm^2} \PR(Y<\lceil j/b\rceil),
\end{align}
where we use that $j$ (and hence $s$) is small relative to $m$.

Hence we only need to check the boundary case:
\begin{multline}
\PR(\pi_1=1)=\cdots=\PR(\pi_1=a)=\\ \frac{b}{(a+mb)^{m-1}}\sum_{s=0}^{m-1} \binom{m-1}{s} (a+sb)^{s-1} ((m-s)b)^{m-2-s} \sim \frac{1}{(b+c)m},
\end{multline}
where we apply the tree function method in the asymptotics.
\end{proof}

\begin{corollary}\label{boundary1-special}
Fix $j$ and take $m$ large relative to $j$. Take $b\geq 1$ any integer. For parking function $\piR$ chosen uniformly at random from $\PF(b, b, m)$, we have
\begin{equation}
\PR(\pi_1=1)=\cdots=\PR(\pi_1=b) \sim \frac{2}{bm},
\end{equation}
and
\begin{equation}
\PR(\pi_1=b+j)\sim \frac{1+\QR_1(\lceil j/b\rceil+1)}{bm},
\end{equation}
where $\QR_{1}(l) = \PR_{1}(X\geq l)$ is the tail distribution function of Borel-$1$ and $\lceil x \rceil$ denotes the smallest integer $\geq x$.
\end{corollary}

\begin{proof}
If $\pi_1=b+j$, then $A_{\pi_2, \dots, \pi_m}=[(s+1)b]$ for some $s$ so that $sb\geq j$. From Proposition \ref{1st}, this implies that
\begin{align}
&\PR(\pi_1=b+j)=\PR(\pi_1=b)-\sum_{s=0}^{\lceil j/b\rceil-1}\PR(A_{\pi_2, \dots, \pi_m}=[(s+1)b]) \notag \\
&=\PR(\pi_1=b)-\sum_{s=0}^{\lceil j/b\rceil-1} \frac{b \binom{m-1}{s} ((s+1)b)^{s-1} ((m-s)b)^{m-2-s}}{((m+1)b)^{m-1}} \notag \\
& \sim \PR(\pi_1=b)-\frac{1}{bm} \left(1-\QR_1(\lceil j/b\rceil+1)\right),
\end{align}
where we use that $j$ (and hence $s$) is small relative to $m$.

Hence we only need to check the boundary case:
\begin{multline}
\PR(\pi_1=1)=\cdots=\PR(\pi_1=b)=\\ \frac{b}{((m+1)b)^{m-1}}\sum_{s=0}^{m-1} \binom{m-1}{s} ((s+1)b)^{s-1} ((m-s)b)^{m-2-s} \sim \frac{2}{bm},
\end{multline}
where we apply Abel's binomial theorem in the asymptotics.
\end{proof}

\end{document}